\documentclass[]{siamart0516}

\usepackage{amsmath, amssymb, amsfonts}
\usepackage[english]{babel}
\usepackage[utf8]{inputenc}
\usepackage{mathrsfs}                     
\usepackage{bbm}                          
\usepackage[all]{xy}                      
\usepackage{graphicx}                     
\usepackage{color}
\usepackage{url}
\usepackage{import}                       
\usepackage{todonotes}                    
\usepackage{algpseudocode}                
\usepackage{enumerate}                    
\usepackage{color}
\usepackage{epstopdf}
\ifpdf
  \DeclareGraphicsExtensions{.eps,.pdf,.png,.jpg}
\else
  \DeclareGraphicsExtensions{.eps}
\fi

\newcommand{\TheTitle}{Optimal Reparametrisations in the Square Root Velocity Framework}
\newcommand{\TheAuthors}{Martins Bruveris}

\headers{Optimal Reparametrisations}{\TheAuthors}

\title{{\TheTitle}\thanks{This work has been supported by the programme ``Infinite-Dimensional Riemannian Geometry with Applications to Image Matching and Shape Analysis'' held at the
Erwin Schrödinger Institute.}}

\author{
  Martins Bruveris\thanks{Department of Mathematics, 
    Brunel University London, Uxbridge UB8 3PH, United Kingdom
    (\email{martins.bruveris@brunel.ac.uk}, 
     \url{http://www.brunel.ac.uk/~mastmmb}).}
}

\usepackage{amsopn}

\ifpdf
\hypersetup{
  pdftitle={\TheTitle},
  pdfauthor={\TheAuthors}
}
\fi

\theoremstyle{plain}
\newtheorem*{thm*}{Theorem}

\theoremstyle{definition}
\newtheorem{remark}[theorem]{Remark}
\newtheorem*{openquestion*}{Open Question}

\def\be{\beta}
\def\ga{\gamma}
\def\de{\delta}
\def\ep{\varepsilon}

\def\th{\theta}

\def\la{\lambda}
\def\rh{\varrho}
\def\si{\sigma}
\def\ta{\tau}
\def\ph{\varphi}

\def\Ga{\Gamma}

\def\o{\circ}
               
\def\inv{^{-1}}
\def\x{\times}

\def\R{{\mathbb R}}

\def\one{\mathbbm{1}}

\let\on=\operatorname

\let\ol=\overline

\let\mb=\mathbb
\let\mc=\mathcal

\newcommand{\ud}{\,\mathrm{d}}

\def\todomartins#1{}

\newcommand{\executeiffilenewer}[3]{%
\ifnum\pdfstrcmp{\pdffilemoddate{#1}}%
{\pdffilemoddate{#2}}>0%
{\immediate\write18{#3}}\fi%
}

\begin{document}

\maketitle

\begin{abstract}
The square root velocity framework is a method in shape analysis to define a distance between curves and functional data. Identifying two curves if they differ by a reparametrisation leads to the quotient space of unparametrised curves. In this paper we study analytical and topological aspects of this construction for the class of absolutely continuous curves. We show that the square root velocity transform is a homeomorphism and that the action of the reparametrisation semigroup is continuous. We also show that given two $C^1$-curves, there exist optimal reparametrisations realising the minimal distance between the unparametrised curves represented by them. Furthermore we give an example of two Lipschitz curves, for which no pair of optimal reparametrisations exists.
\end{abstract}


\section{Introduction}

In this paper we want to analyse a variational problem that arises in the context of shape analysis. By \emph{shape} we mean parametrised curves of a given regularity class, with two curves identified if they differ by a translation or a reparametrisation. Denote by $B(I,\R^d)$ the \emph{shape space}, i.e., the set of all shapes, $I$ being an interval.

The goal of shape analysis is to compare, classify and identify shapes, describe the variability of a class of shapes and to quantify the information contained within a shape. The basis for these operations is provided by a distance function on shape space. There are many distance functions to choose from. A distance arising as the geodesic distance of a Riemannian metric provides additional properties to the shape space: the structure of a smooth manifold, the exponential map and minimal geodesics realising the distance can all be exploited in applications~\cite{Pennec2006b}.

Riemannian metrics on the space of curves and on the shape space of un\-pa\-ra\-metrised curves have been studied in~\cite{Michor2007,Michor2006c,Younes1998,Mio2007} as well as many later papers; an overview can be found in \cite{Bauer2014}.

A Riemannian distance that is particularly well-suited for applications is the one used in the square root velocity framework \cite{Srivastava2011}. We assign each curve $c$ its \emph{square root velocity function} (SRVF), $q = \frac{c'}{\sqrt{|c'|}}$, with the convention that $q(t) = 0$, if $c'(t) = 0$. The distance between two curves is then
\begin{equation}
\label{eq:dist_srvt}
\on{dist}(b,c) = \| p - q \|_{L^2} = \left\| \frac{b'}{\sqrt{|b'|}} - \frac{c'}{\sqrt{|c'|}} \right\|_{L^2}\,,
\end{equation}
where $p = \frac{b'}{\sqrt{|b'|}}$ is the SRVF of $b$. The natural space on which to define this distance is $AC_0(I,\R^d)$, the space of absolutely continuous curves with $c(0) = 0$. We will discuss in Sect.~\ref{sec:srvt_geometry} how this distance relates to a Riemannian metric.

Equation~\eqref{eq:dist_srvt} defines a distance on the space of parametrised curves and on shape space we consider the corresponding quotient distance,
\begin{equation}
\label{eq:dist_quot}
\on{dist}([b], [c]) = \inf_{\be, \ga \in \ol\Ga} \on{dist}(b \o \be, c \o \ga)\,;
\end{equation}
here $\ol\Ga$ is the semigroup of weakly increasing, surjective, absolutely continuous maps $\be, \ga: I \to I$ and we identify shape space with the quotient $B(I,\R^d) = AC_0(I,\R^d)/\ol\Ga$.\footnote{Since $\ol\Ga$ is a semigroup and not a group, this is not entirely correct. In fact $B(I,\R^d)$ consists of closures of $\ol\Ga$-orbits. See Sect.~\ref{sec:quotient_space} for details.}

The square root velocity framework for curves traces its origin to~\cite{Srivastava2011}. It has been used to analyse the shape of plant leaves~\cite{Laga2014} and arteries~\cite{Xie2014}, to segment handwritten text~\cite{Kurtek2014}, to globally align RNA sequences~\cite{Laborde2013}, to perform statistical analysis of manual image segmentations~\cite{Kurtek2013} and to study the shape of the corpus callosum in schizophrenic patients~\cite{Joshi2013}. The framework has been generalised to manifold-valued data and it has been used to analyse migration patterns of birds~\cite{Su2014} and audio-visual speech recognition~\cite{Su2014b}.

For scalar-valued data, i.e. $d=1$, the square root velocity framework is closely related to the Fisher--Rao metric on the space of probability densities~\cite{Tucker2013, Amari2000}. The framework has been used to align chromatograms~\cite{Wallace2014}, analyse proteomics data~\cite{Tucker2014, Cheng2014} and SONAR signals~\cite{Tucker2014b}.

In all these applications the distance~\eqref{eq:dist_srvt} on the quotient space $B(I,\R^d)$ plays an important role. For some algorithms, e.g. the computation of the Karcher mean of a set of shapes~\cite{Laga2014}, one requires not only the numerical value of the distance, but also the reparametrisations $\be, \ga$ realising
\[
\on{dist}([b], [c]) =  \on{dist}(b \o \be, c \o \ga)\,.
\]
The optimal reparametrisations $\be, \ga$ then describe an alignment between the curves $b,c$ in the sense that the point $\be(t)$ on $b$ corresponds to the point $\ga(t)$ on $c$. The question, whether these optimal reparametrisations exist, is not trivial. The best result until now is that optimal reparametrisations exist, if one of the curves is piecewise linear~\cite{Lahiri2015_preprint}.

We will show the following theorem regarding the (non-)existence of optimal reparametrisations.
\begin{thm*}
Let $d \geq 1$.
\begin{enumerate}
\item
If $b,c \in C^1(I,\R^d)$, then there exist $\be, \ga \in \ol\Ga$ realising the infimum in~\eqref{eq:dist_quot}.
\item
If $d \geq 2$, there exists a pair of Lipschitz curves for which the infimum in~\eqref{eq:dist_quot} is not realised by any pair of reparametrisations.
\end{enumerate}
\end{thm*}
The first part of the theorem is Prop.~\ref{prop:ex_opt_reparam} and the second part is Cor.~\ref{cor:count_ex}. Before we arrive at these results, we will discuss in Sect.~\ref{sec:srvt} the continuity of the square root velocity transform $R(c) = \frac{c'}{|c'|}$ on spaces of curves of finite regularity and in Sect.~\ref{sec:semigroup} and Sect.~\ref{sec:shape_space} properties of the $\ol\Ga$-action and the topology of the orbits.

\subsection*{Notation} For the purposes of this paper we set $I=[0,1]$.

We call a $C^1$-curve $c$ \emph{regular}, if $c'(t) \neq 0$ holds for all $t \in I$. Similarly, an absolutely continuous curve $c$ is regular, if $c' \neq 0$ holds a.e.. Note that this means that a curve can be regular as an absolutely continuous curve and non-regular as a $C^1$-curve. It will be clear from the context, which notion of regularity is used.

\section{Riemannian geometry of the square root velocity framework}
\label{sec:srvt_geometry}

In this section we want to discuss how the distance~\eqref{eq:dist_srvt} is connected to a Riemannian metric on the space of smooth, regular curves.

\subsection{Smooth curves}
For now assume $d \geq 2$ and denote by
\[
\on{Imm}(I,\R^d) = \{ c \in C^\infty(I,\R^d) \,:\, c'(t) \neq 0\; \forall t \in I \}
\]
the space of immersions. On this space we define
\begin{equation}
\label{eq:our_metric}
G_c(h,k) = \int_I \langle D_s h^\perp, D_s k^\perp \rangle + 
\frac 14 \langle D_s h, D_s c \rangle \langle D_s k, D_s c \rangle \ud s\,.
\end{equation}
Here $h,k \in T_c \on{Imm}(I,\R^d)$ are elements of the tangent space; immersions form an open subset of $C^\infty(I,\R^d)$ and thus $h,k \in C^\infty(I,\R^d)$; geometrically they are vector fields along the curve. We denote by $D_s h = \frac{1}{|c'|} h'$ and $\ud s = |c'| \ud \th$ differentiation and integration with respect to arc length, $D_s c = \frac{c'}{|c'|}$ is the unit length tangent vector along $c$ and $D_s h^\perp = D_s h - \langle D_s h, D_s c \rangle D_s c$ is the projection of $D_s h$ to the subspace $\{ D_s c\}^\perp$ orthogonal to the curve. Differentiation and integration with respect to arc length are used to make $G$ invariant with respect to reparametrisations.

Thus defined, $G$ is a Riemannian metric on the space $\on{Imm}_0(I,\R^d)$ of curves starting at the origin, $c(0) = 0$; this space can be identified with the quotient of $\on{Imm}(I,\R^d)$ by the group of translations.

What distinguishes the metric $G$ defined in~\eqref{eq:our_metric} among other possible choices is the existence of the square root velocity transform,
\[
R: \operatorname{Imm}_0(I,\mathbb R^d) \rightarrow C^{\infty}(I,\mathbb R^d \setminus \{0\} )\,,\quad
c \mapsto \frac{1}{\sqrt{|c'|}} c'\;.
\]
We equip $C^\infty(I, \R^d \setminus \{0\})$ with the $L^2$-inner product, viewed as a constant (and hence flat) Riemannian metric. Geodesics in $C^\infty(I,\R^d \setminus \{0\})$ with respect to this Riemannian metric correspond to pointwise geodesics in $\R^d\setminus \{0\}$ in the following sense: a path $s \mapsto q(s,\cdot)$ is a geodesic in $C^\infty(I,\R^2\setminus \{0\})$ if and only if for all $t \in I$ the curve $s \mapsto q(s,t)$ is a geodesic in $\R^d\setminus \{0\}$.

\begin{theorem}
The following holds:
\begin{enumerate}
\item
$R$ is a diffeomorphism between $\on{Imm}_0(I,\R^d)$ and $C^\infty(I,\R^d\setminus\{0\})$;
\item
$R$ is an isometry between $(\on{Imm}_0(I,\R^d), G)$ and $(C^\infty(I, \R^d\setminus \{0\}), L^2)$.
\end{enumerate}
\end{theorem}

Let us briefly sketch the proof, which can be found for example in \cite{Bauer2014c}. The inverse of $R$ is given by
\[
R\inv(q)(t) = \int_0^t q|q|\ud \ta\,,
\]
allowing us to verify that $R$ is a diffeomorphism. In order for $R$ to be an isometry we need the relation
\begin{equation}
\label{eq:defining_isometry}
G_c(h,k) = \langle DR(c).h, DR(c).k \rangle_{L^2}
\end{equation}
to hold. The derivative of $|c'|^{-1/2}$ is $D\left(|c'|^{-1/2}\right).h = -\frac 12 |c'|^{-1/2} \langle D_s h, D_s c\rangle$ and hence the derivative of $R$ can be expressed as
\[
DR(c).h = \left( D_s h - \frac 12 \langle D_s h, D_s c \rangle D_s c \right) \sqrt{|c'|}\,.
\]
With this formula it is easy to check that \eqref{eq:defining_isometry} holds. We also see the reason for the appearance of the factor $\frac 14$ in~\eqref{eq:our_metric}.

That $R$ is an isometry means that at least locally the geodesic distance between two curves $b,c \in \on{Imm}_0(I,\R^d)$ is given by
\begin{equation*}
\on{dist}(b,c) = \| R(b) - R(c) \|_{L^2} = \left\| \frac{b'}{\sqrt{|b'|}} - \frac{c'}{\sqrt{|c'|}} \right\|_{L^2}\,,
\end{equation*}
which is exactly the distance~\eqref{eq:dist_srvt}. The global behaviour of the distance depends on the dimension $d$ of the ambient space.

Assume $d \geq 3$. The space $C^\infty(I,\R^d\setminus\{0\})$ is not convex and since geodesics are straight lines, it is also not geodesically convex. Nevertheless we are able to smoothly perturb any path that passes through the origin in such a way that the perturbed path avoids the origin and hence the geodesic distance on all of $\on{Imm}_0(I,\R^d)$ is given by~\eqref{eq:dist_srvt}.

\subsection{Plane curves}
The perturbation argument does not work for plane curves, i.e., for $d=2$. We can see in Fig.~\ref{fig:2d_geodesic} two curves, for which $\| R(b) - R(c)\|_{L^2}$ does \emph{not} represent the geodesic distance, because the straight line $(1-t)R(b) + tR(c)$ connecting them leaves $C^\infty(I,\R^2\setminus \{0\})$. What we can do however is to extend geodesics across the origin in $C^\infty(I, \R^d\setminus\{0\})$ and obtain $C^\infty(I,\R^d)$ as the \emph{geodesic completion}, i.e., a geodesically complete manifold containing $C^\infty(I, \R^d\setminus\{0\})$ as an isometric, totally geodesic submanifold. This allows us to interpret~\eqref{eq:dist_srvt} as the geodesic distance on the geodesic completion.

\begin{figure}
\centering
\def\svgwidth{.40\columnwidth}
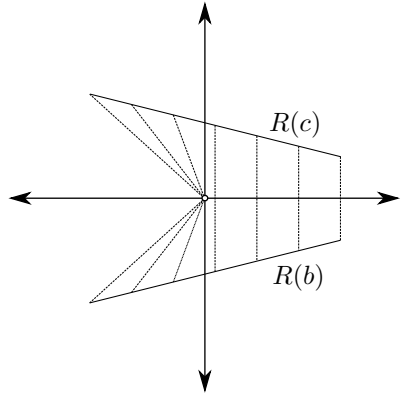
\caption{When $R(b)$ and $R(c)$ are as above, then the straight line $(1-t)R(b) + tR(c)$ between them passes through the origin. To realize the geodesic distance in $C^\infty(I,\R^2\setminus \{0\})$ one requires a path similar to the one shown in the figure. While the path itself leaves $C^\infty(I,\R^2\setminus \{0\})$, it can be approximated by paths, that avoid the origin and thus remain inside $C^\infty(I,\R^2\setminus \{0\})$.
}
\label{fig:2d_geodesic}
\end{figure}

We have to be careful with this interpretation. It is easy to extend geodesics in the space of SRVFs, that is on the image side of $R$. We can also extend $R\inv$, given by
\[
R\inv(q)(t) = \int_0^t q |q| \ud \ta\,,
\]
from $C^\infty(I, \R^d\setminus\{0\})$ to $C^\infty(I,\R^d)$ and we have $R\inv(C^\infty(I,\R^d)) = C^\infty_0(I,\R^d)$ with $C^\infty_0(I,\R^d) = \{ c \in C^\infty \,:\, c(0) = 0 \}$. However the extended map is not a diffeomorphism any more: if a function $q$ passes through the origin, then $DR\inv(q)$ is not surjective. Thus the geodesic completion of $\on{Imm}_0(I,\R^d)$ is $C^\infty_0(I,\R^d)$ as a set, but with the differential structure, that is induced by $R\inv$.

\subsection{Scalar functions}

We can perform the same construction when $d=1$. For functions $c: I \to \R$, the interval $I$ often parametrises time and $c$ itself represents functional, i.e. time-dependent, data. Because of the importance of functional data in applications we want to describe the above construction in this particular case.

After removing the origin, $\R$ becomes disconnected and thus the space of regular curves becomes
\[
\on{Imm}_0(I,\R) = \{ c \in C^\infty_0(I,\R) \,:\, c' > 0 \text{ or } c' < 0 \}
\]
the set of strictly increasing and decreasing functions. For the Riemannian metric we have $D_s h^\perp = 0$ and so
\[
G_c(h,k) = \frac 14 \int_I h'k'\, |c'|\inv \ud \th\,.
\]
The square root velocity transform is simply
\[
R(c) = \sqrt{|c'|}\,.
\]
Each connected component of $\on{Imm}_0(I,\R)$ is convex and thus the geodesic distance on each one is given by
\[
\on{dist}(b,c) = \left\| \sqrt{|b'|} - \sqrt{|c'|} \right\|_{L^2}\,.
\]
The space of SRVFs is $C^\infty(I,\R\setminus \{0\})$, which is disconnected as well and we can consider $C^\infty(I,\R)$ as its geodesic completion. Note that this choice is not unique: we could also take two copies of $C^\infty(I,\R)$---one for positive and one for negative SRVFs---as the geodesic completion. However we choose to glue the two connected components together. For the curves itself this means that the set $C^\infty_0(I,\R)$ is the geodesic completion of $\on{Imm}_0(I,\R)$ and~\eqref{eq:dist_srvt} is the geodesic distance on $C^\infty_0(I,\R)$. As for plane curves we make no statement about the differentiable structure of the geodesic completion.

\subsection{Metric completion}

The metric completion of $C^\infty(I,\R^d)$ with respect to the $L^2$-distance is $L^2(I,\R^d)$. To see what class of curves this corresponds to, we look at the formula for $R\inv$,
\[
R\inv(q)(t) = \int_0^t q |q| \ud \ta\,,
\]
and we see that if $q \in L^2$, then $R\inv(q)$ is an absolutely continuous curve. In fact we can extend $R$ to a bijective map
\[
R : AC_0(I,\R^d) \to L^2(I,\R^d)\,,
\]
where $AC_0(I,\R^d)$ is the set of absolutely continuous curves $c:I\to \R^d$ with $c(0) = 0$. We will show in Sect.~\ref{sec:srvt} that this map is a homeomorphism, but not differentiable.

Since for $d \geq 3$ the distance~\eqref{eq:dist_srvt} is the geodesic distance on $\on{Imm}_0(I,\R^d)$, it follows that $(AC_0(I,\R^d), \on{dist})$ is the metric completion of the Riemannian manifold $(\on{Imm}_0(I,\R^d), G)$ with $G$ given by~\eqref{eq:our_metric}. Similarly, for $d=1,2$ the space $AC_0(I,\R^d)$ is the metric completion of the geodesic completion of $(\on{Imm}_0(I,\R^d), G)$.

\subsection{Higher order metrics}
The metric $G$ belongs to the family of Sobolev type metrics. Since it involves first order derivatives of the tangent vectors, it is a first order metric. A more general Sobolev type metric of order $n$ is one of the form
\[
G^n_c(h,k) = \int_I a_0 \langle h, k \rangle + a_1 \langle D_s h, D_s k \rangle + \dots
+ \langle D_s^n h, D_s^n k \rangle \ud s\,,
\]
with constants $a_j$. For closed curves first order metrics have been studied in \cite{Michor2008a, Mio2007, Sundaramoorthi2011, Bauer2014c} and higher order metrics in \cite{Michor2007,Mennucci2008,Bruveris2014,Bruveris2015}.

It is instructive to compare the behaviour of $G$, which is a first order metric, to higher order Sobolev metrics as well as to the $L^2$-metric. We will talk about closed curves here, since most references only treat closed curves. For the $L^2$-metric the picture is simple: the geodesic distance on $\on{Imm}(S^1,\R^d)$, induced by the reparametrisation invariant $L^2$-metric is identically zero \cite{Michor2005,Bauer2012c} and hence there is no completion worth speaking of.

For Sobolev type metrics of order $n \geq 2$ the completion of the space of smooth, regular curves is $\{ c \in H^n \,:\, c(t) \neq 0\}$ the space of regular curves of Sobolev order $n$. Two differences jump out: for the first-order metric the completion leaves the class of $L^2$-based Sobolev spaces and the completion contains non-regular curves, e.g. the constant curve. A more detailed comparison of different Sobolev metrics can be found in \cite{Bauer2015_preprint,Bauer2014}.

\section{Extending the square root velocity transform}
\label{sec:srvt}

We want to extend the square root velocity transform to spaces larger than the space of smooth, regular curves.

\subsection{Lipschitz curves}

First we note that the pointwise map $x \mapsto \frac{x}{\sqrt{|x|}}$ is continuous on $\R^d$ and H\"older continuous with exponent $\frac 12$. The proof of this lemma is standard.

\begin{lemma}
\label{lem:V_holder}
The map $V : \R^d \to \R^d$ defined by $V(x) = \frac{x}{\sqrt{|x|}}$ and $V(0) = 0$ is continuous and H\"older continuous with exponent $\frac 12$.
\end{lemma}

From Lem.~\ref{lem:V_holder} it immediately follows that we can extend the square root velocity transform to Lipschitz curves.

\begin{corollary}
The map $R : W^{1,\infty}(I,\R^d) \to L^{\infty}(I,\R^d)$ defined by $R(c) = V \o c'$ is continuous and H\"older continuous with exponent $\frac 12$.
\end{corollary}

\begin{proof}
Take the supremum in
\[
\left|R(c_1)(t) - R(c_2)(t)\right|
= \left| V(c_1'(t)) - V(c_2'(t)) \right|
\leq C \left| c_1'(t) - c_2'(t) \right|^{1/2}\,,
\]
with $C$ being the H\"older constant of $V$.
\end{proof}

Similarly it can be shown that for all $k \geq 1$ the maps
\[
R : C^k(I,\R^d) \to C^{k-1}(I,\R^d)
\]
are continuous and H\"older continuous with exponent $\frac 12$.

\subsection{Absolutely continuous curves}

We are mostly interested in the class of absolutely continuous curves, since these form the metric completion of the Riemannian manifold of smooth, regular curves. Equip the space $AC(I,\R^d)$ of absolutely continuous curves with the norm $\| c\|_{AC} = |c(0)| + \| c' \|_{L^1}$; with this norm $(AC, \|\cdot\|_{AC})$ and $(AC_0, \|\cdot\|_{AC})$, the subspace of curves with $c(0)=0$, are Banach spaces.

We can extend the square root velocity transform to
\[
R : AC_0(I,\R^d) \to L^2(I,\R^d),\, R(c)(t) = V(c'(t))
\]
This is well-defined, since
\[
\| R(c) \|^2_{L^2} = \int_I \frac{1}{|c'|} |c'|^2 \ud t = \int_I |c'| \ud t = \| c \|_{AC}\,.
\]
Note that $R$ is not an isometry, but it preserves the norm in the sense that $\| R(c) \|^2_{L^2} = \| c \|_{AC}$. The transform is bijective and the inverse is
\[
R\inv(q)(t) = \int_0^t q |q| \ud \ta\,.
\]
Continuity of $R$ on $AC_0$ is not as easy to show as for Lipschitz curves and it is not known, whether $R$ is H\"older continuous on $AC_0$.

\begin{lemma}
\label{lem:R_cont}
The map $R$ is a homeomorphism between $AC_0(I,\R^d)$ and $L^2(I,\R^d)$.
\end{lemma}

\begin{proof}
The continuity of $R\inv$ is simple. Let $q_n \to q$ in $L^2$ and set $c_n = R\inv(q_n)$. Then
\begin{align*}
\| c_n - c\|_{AC} 
&= \| c_n' - c' \|_{L^1} = \int_I \big| q_n |q_n| - q |q| \big| \ud t \\
&\leq \int_I \big| q_n |q_n| - q_n |q| \big| + \big| q_n |q| - q|q| \big| \ud t \\
&\leq \int_I |q_n| \cdot |q_n - q| + |q_n - q| \cdot |q| \ud t\,,
\end{align*}
and using Cauchy--Schwartz we obtain the convergence $c_n \to c$ in $AC_0$.

Next we show that $R$ is continuous. Let $c_n \to c$ in $AC_0$. Then $c_n' \to c'$ in $L^1$ and by \cite[Satz VI.5.4]{Elstrodt2009} we also have convergence $c_n' \to c'$ locally in measure. We will use that a sequence converges against a limit, if every subsequence has a subsequence converging against the same limit. Assume a subsequence has been chosen. Then by \cite[Satz VI.4.14]{Elstrodt2009} this subsequence of $(c_n')_{n \in \mathbb N}$ has a subsequence $(c_{n_k})_{k \in \mathbb N}$, converging $c'_{n_k} \to c'$ almost everywhere. The map $V(x) = \frac{x}{\sqrt{|x|}}$ is continuous on $\R^d$ and so with the notation $q_n = R(c_n)$ we have $q_{n_k} \to q$ a.e. as well. Now by the above calculation we also have $\|q_{n_k}\|^2_{L^2} = \|c_{n_k}\|_{L^1} \to \| c\|_{L^1} = \|q\|^2_{L^2}$ and by \cite[Korollar VI.5.5]{Elstrodt2009} convergence a.e. together with convergence of the norms implies $q_{n_k} \to q$ in $L^2$. Since this holds for every subsequence we also obtain $q_n \to q$ in $L^2$, thus showing the continuity of $R$.
\end{proof}

\subsection{Differentiability}

The square root velocity transform can be extended to a continuous map on absolutely continuous curves, but by doing so we loose differentiability properties. In particular we have the following result.

\begin{proposition}
Let $c \in AC_0(I,\R^d)$ and $q \in L^2(I,\R^d)$ and assume that $c'=0$ and $q=0$ on sets of positive measure. The map $R : AC_0(I,\R^d) \to L^2(I,\R^d)$ has the following properties:
\begin{enumerate}
\item
\label{diff_p1}
$R$ is not differentiable at $c$;
\item
\label{diff_p2}
$DR\inv(q)$ is not surjective;
\item
\label{diff_p3}
$R\inv$ is not twice differentiable at $q$.
\end{enumerate}
\end{proposition}

\begin{proof}
Take a curve $c \in AC_0(I,\R^d)$, such that $c'=0$ on a set of positive measure and let $h \in AC_0(I,\R^d)$ be a function with $\on{supp} h' \subseteq \{ c'=0 \}$. Then $c'(t)+\ep h'(t) = c'(t)$ for all $t$ with  $c'(t) \neq 0$ and we have
\[
\frac 1\ep \left(R(c+\ep h) - R(c) \right) = \frac 1\ep \frac{\ep h'}{\sqrt{|\ep h'|}} = \ep^{-1/2} \frac{h'}{\sqrt{|h'|}}\,,
\]
and we see that $R$ is not differentiable at $c$.

For the second part we calculate the derivative $DR\inv$,
\[
DR\inv(q).h(t) = \int_0^t h|q| + \frac{q}{|q|} \langle q, h \rangle \ud \ta\,.
\]
Let $v = DR\inv(q).h$ and we see that $v'$ vanishes wherever $q$ vanishes. If $q=0$ on a set of positive measure, then $DR\inv(q)$ cannot be surjective.

With $q, h \in L^2(I,\R^d)$ and $q=0$ on a set of positive measure, choose $k \in L^2(I,\R^d)$ with $\{k \neq 0\} \subseteq \{q = 0\}$. Then we have
\begin{align*}
\frac{1}{\ep}\left(DR\inv(q + \ep k).h - DR\inv(q).h\right)(t) 
&= \frac 1\ep \int_0^t h|\ep k| + \frac{\ep k}{|\ep k|} \langle \ep k, h \rangle \ud \ta \\
&= \frac {|\ep|}{\ep} \int_0^t h|k| + \frac{k}{|k|} \langle k, h \rangle \ud \ta\,,
\end{align*}
and we see that the limit $\ep \to 0$ does not exist, since it depends on the sign of $\ep$. Thus $R\inv$ is not twice differentiable at $q$.
\end{proof}

Similar lack of differentiability properties hold for the square root velocity transform on smooth, nonregular curves: the map $R: C^\infty_0(I,\R^d) \to C^\infty(I,\R^d)$ is not differentiable at any curve, whose derivative vanishes at some point; the inverse is not twice differentiable at $q$ and $DR\inv(q)$ is not surjective, if $q$ passes through the origin.

\subsection{Geodesic distance} 
\label{sec:geod_dist}
We define a distance on $AC_0(I,\R^d)$ via
\[
\on{dist}(b, c) = \| R(b) - R(c) \|_{L^2}\,.
\]
This is an extension of the geodesic distance on $\on{Imm}_0(I,\R^d)$ as discussed in Sect.~\ref{sec:srvt_geometry}. Because of Lem.~\ref{lem:R_cont}, the topologies induced by $\|\cdot\|_{AC}$ and $\on{dist}$ on $AC_0$ coincide. By construction $R$ is an isometry between the metric spaces $(AC_0(I,\R^d), \on{dist})$ and $(L^2(I,\R^d), \|\cdot\|_{L^2})$. In particular $(AC_0(I,\R^d), \on{dist})$ is a complete metric space.

\section{Semigroup of reparametrisations}
\label{sec:semigroup}

We are dealing with $AC_0(I,\R^d)$, the space of absolutely continuous curves starting at the origin, and thus the natural group of reparametrisations is
\[
\Ga = \{ \ga : I \to I \,:\, \ga \text{ abs. cont.},\, 
\ga(0) = 0,\, \ga(1) = 1,\, 
\ga' > 0 \text{ a.e.}\}\,,
\]
the group of absolutely continuous homeomorphisms. It will be necessary to also consider the semigroup
\[
\ol\Ga = \{ \ga : I \to I \,:\, \ga \text{ abs. cont.},\, 
\ga(0) = 0,\, \ga(1) = 1,\, 
\ga' \geq 0 \text{ a.e.}\}\,,
\]
consisting of weakly, increasing absolutely continuous functions. Both $\Ga$ and $\ol\Ga$ are subsets of $AC_0(I,\R)$ and we endow them with the induced topology. With this topology $\ol \Ga$ coincides with the closure of $\Ga$ in $AC_0(I,\R)$. The semigroup $\ol \Ga$ acts on $AC_0(I,\R^d)$ from the right via $(c, \ga) \mapsto c \o \ga$ and the action is by isometries, as can be seen from
\[
\| c \o \ga \|_{AC} = \int_I |c' \o \ga| \ga' \ud t = \int_I |c'| \ud t = \| c\|_{AC}\,.
\]
Here we used a general form of the change of variables formula; see e.g. \cite[(20.5)]{Hewitt1975}. 

\begin{openquestion*}
Is $\Ga$ a topological group? The continuity of the multiplication follows from Prop.~\ref{prop:action_cont}, but the continuity of the inversion map $\ga \mapsto \ga\inv$ is not clear.
\end{openquestion*}

We will be concerned with the orbits in $AC_0$ of the $\Ga$- and $\ol\Ga$-actions, since these orbits will correspond to unparametrised curves. In this section we prepare for the study of the orbit space by showing that the $\ol\Ga$-action on $AC_0$ is continuous and by identifying the closure of $\Ga$-orbits.  Before we prove these results, we need a lemma about the continuity of the inversion on $\Ga$.

\begin{lemma}
\label{lem:inv_conv_ga}
If $\ga_n \in \ol \Ga$, $\de_n \in \Ga$ and $\ga_n - \de_n \to 0$ in $AC_0$, then $\ga_n \o \de_n\inv \to \on{Id}$ in $\ol \Ga$.
\end{lemma}

\begin{proof}
We have to show that $(\ga_n \o \de_n\inv)' \to 1$ in $L^1$. First we note that
\[
(\ga_n \o \de_n\inv)' - 1 
= \left(\ga_n' \o \de_n\inv - \frac{1}{(\de_n\inv)'}\right) (\de_n\inv)'
= \left(\ga_n' \o \de_n\inv - \de_n' \o \de_n\inv\right) (\de_n\inv)'
\]
Now we integrate this and obtain
\begin{align*}
\| (\ga_n \o \de_n\inv)' - 1 \|_{L^1}
&= \int_I \left| \ga_n' \o \de_n\inv - \de_n' \o \de_n\inv\right| (\de_n\inv)' \ud t
= \int_I \left| \ga_n' - \de_n' \right| \ud t\,.
\end{align*}
We can use the change of variables, because $\de_n\inv \in \Ga$ and from here the statement of the lemma follows.
\end{proof}

Now we can proceed with the main proposition.

\begin{proposition}
\label{prop:action_cont}
The action of $\ol \Ga$ on $AC_0(I,\R^d)$ is continuous.
\end{proposition}

\begin{proof}
The proof will proceed in three steps. First we consider the action of $\ol \Ga$ on a fixed, piecewise linear curve around $\on{Id} \in \ol \Ga$, then the action on a general curve and finally the joint continuity of the map $(c, \ga) \mapsto c \o \ga$.

{\bfseries Step 1:} piecewise linear curves, continuity at $\on{Id} \in \ol \Ga$. \\
Let $c \in AC_0(I,\R^d)$ be a piecewise linear curve, i.e., $c' = \sum_{j=1}^N a_j \one_{I_j}$ with $I_j = [t_{j-1},t_j]$ and $0 = t_0 < t_1 < \dots < t_N = 1$. Take a sequence $\ga_n \to \on{Id}$ in $\ol \Ga$. We need to show that $c \o \ga_n \to c$ in $AC_0$. 

Assume that $n$ is large enough, such that $\ga_n(I_j) \subseteq I_{j-1} \cup I_j \cup I_{j+1}$, in other words, $\ga_n([t_{j-1}, t_j]) \subseteq [t_{j-2}, t_{j+1}]$. Define the three sets
\[
A_{j,n}^- = \ga_n(I_j) \cap I_{j-1},\quad
A_{j,n} = \ga_n(I_j) \cap I_{j},\quad
A_{j,n}^+ = \ga_n(I_j) \cap I_{j+1}\,,
\]
which form a decomposition of each interval $I_j = A_{j,n}^- \cup A_{j,n} \cup A_{j,n}^+$. Then
\begin{align*}
\| c &- c\o\ga_n\|_{AC} = \int_I \left|c' - (c'\o\ga_n)\ga_n'\right| \ud t  \\
&\leq \sum_{j=1}^N \int_{A_{j,n}^-} |a_j - a_{j-1} \ga_n'| \ud t
+ \int_{A_{j,n}} |a_j| \cdot |1 - \ga_n'| \ud t
+ \int_{A_{j,n}^+} |a_j - a_{j+1} \ga_n'| \ud t\,.
\end{align*}
As $\ga_n \to \on{Id}$ uniformly, it follows that $\la(A_{j,n}^-) \to 0$ and $\la(A_{j,n}^+) \to 0$, where $\la$ denotes the Lebesgue measure, and hence the first and third integrals converge to 0. For the second integrals we have
\[
\int_{A_{j,n}} |a_j| \cdot |1 - \ga_n'| \ud t \leq \| c'\|_{\infty} \| \on{Id} - \ga_n \|_{}\,,
\]
and we see that $c \o \ga_n \to c$ in $AC_0$.

{\bfseries Step 2:} fixed arbitrary curve, continuity at $\on{Id} \in \ol \Ga$. \\
Let $c \in AC_0(I,\R^d)$ and take a sequence $\ga_n \to \on{Id}$ in $\ol \Ga$. Let $\ep > 0$ be given. Piecewise linear curves are dense in $AC_0$ and so we can choose a piecewise linear $v$, with $\|c - v\|_{AC} < \frac{\ep}3$. Then
\begin{align*}
\| c - c \o \ga_n \|_{AC} 
&\leq \| c - v \|_{AC} + \| v - v \o \ga_n \|_{AC} + \| v \o \ga_n - c \o \ga_n \|_{AC} \\
&\leq \frac{2\ep} 3 + \| v - v \o \ga_n \|_{AC}\,.
\end{align*}
Using that $\ol \Ga$ acts by isometries and the convergence for step functions shown in Step 1, we conclude that $c\o\ga_n \to c$ in $AC_0$.

{\bfseries Step 3:} joint continuity. \\
Now we take sequences $c_n \to c$ in $AC_0$ and $\ga_n \to \ga$ in $\ol \Ga$ and we want to show that $c_n \o \ga_n \to c \o \ga$ in $AC_0$. Since $\Ga$ is dense in $\ol \Ga$ we can find another sequence $\de_n \in \Ga$ with $\de_n - \ga_n \to 0$ in $AC_0$. Now we estimate
\begin{align*}
\| c_n \o \ga_n &- c \o \ga \|_{AC} \leq \\
&\leq \| c_n \o \ga_n - c \o \ga_n \|_{AC} + \| c \o \ga_n - c \o \de_n \|_{AC}
+ \| c \o \de_n - c \o \ga \|_{AC} \\
&\leq \| c_n - c \|_{AC} + \| c \o \ga_n \o \de_n\inv - c \|_{AC}
+ \| c - c \o \ga \o \de_n\inv \|_{AC}\,.
\end{align*}
By Lem.~\ref{lem:inv_conv_ga} we have $\ga_n \o \de_n\inv \to \on{Id}$ and $\ga \o \de_n \inv \to \on{Id}$ in $\ol \Ga$ and hence $c\o \ga_n \o \de_n\inv \to c$ and $c \o \ga \o \de_n\inv \to c$ in $AC_0$. This concludes the proof.
\end{proof}

Using the continuity of the action we can show that $\ol\Ga$-orbits of regular curves are closed.

\begin{proposition}
\label{prop:orbit_closed}
Let $c \in AC_0(I,\R^d)$ be a curve with $c' \neq 0$ a.e.. Then the orbit $c \o \ol \Ga$ is closed in $AC_0(I,\R^d)$.
\end{proposition}

\begin{proof}
By choosing a constant speed parametrisation we can assume that $|c'| \equiv \la_c$ is constant. Let $c_n = c \o \ga_n$ with $\ga_n \in \ol \Ga$ be a sequence in $c \o \ol \Ga$ and $c_n \to \tilde c$ in $AC_0$. Write $\tilde c = b \o \ga$ with $b$ a constant speed curve, $|b'|\equiv \la_b$ and $\ga \in \ol\Ga$. The identity $|c_n'| = |c' \o \ga_n| \cdot \ga_n' = \la_c \ga_n'$ and $|c'| = \la_b \ga'$ together with the convergence $|c_n'| \to |c'|$ in $L^1$, imply $\la_c \ga_n' \to \la_b \ga'$ in $L^1$. Since $\int_I \la_c \ga_n' = \la_c$, it follows that $\la_c = \la_b$ and hence $\ga_n \to \ga$ in $\ol \Ga$.

It remains to show that $c=b$. As $\Ga$ is dense in $\ol \Ga$, we can choose a sequence $\de_n \in \Ga$, such that $\de_n - \ga_n \to 0$ in $AC_0$. In particular this implies $\de_n \to \ga$ in $\ol \Ga$. The calculation
\[
\int_I \left| \left(c'_n \o \de_n\inv \right) (\de_n\inv)' - \left(\tilde c' \o \de_n\inv\right) (\de_n\inv)' \right| \ud t
= \int_I \left| c'_n - \tilde c' \right| \ud t \to 0
\]
shows that $c_n \o \de_n\inv - \tilde c \o \de_n\inv \to 0$ in $AC_0$. Now $c_n \o \de_n\inv = c \o \ga_n \o \de_n\inv \to c$ in $AC_0$, since $\ga_n \o \de_n\inv \to \on{Id}$ by Lem.~\ref{lem:inv_conv_ga} and using the same lemma also $\tilde c \o \de_n\inv = b \o \ga \o \de_n\inv \to b$. This implies $b=c$ and the proof is complete.
\end{proof}

The above result has an important corollary: the closure of the $\Ga$-orbit of a curve is equal to the $\ol\Ga$-orbit of a regular reparametrisation of it; consequently, if a curve is already regular, then the closure of its $\Ga$-orbit equals its $\ol\Ga$-orbit.

\begin{corollary}
\label{cor:orbit_id}
Let $c \in AC_0(I,\R^d)$.
\begin{enumerate}
\item
\label{cor_item1}
If $c = b \o \ga$ with $b' \neq 0$ a.e., then $\ol{c \o \Ga} = b \o \ol \Ga$.
\item
\label{cor_item2}
If $c'\neq 0$ a.e., then $\ol{c \o \Ga} = c \o \ol\Ga$.
\end{enumerate}
\end{corollary}

\begin{proof}
Clearly $c \o \Ga \subset b \o \ol \Ga$ and since $b \o \ol\Ga$ is closed we have $\ol{c \o \Ga} \subseteq b \o \ol\Ga$. If $b \o \be \in b \o \ol\Ga$, choose sequences $\be_n, \ga_n \in \Ga$ with $\be_n \to \be$, $\ga_n \to \ga$. Then $c \o \ga_n\inv \o \be_n = b \o (\ga \o \ga_n\inv) \o \be_n \to b \o \be$ using Lem.~\ref{lem:inv_conv_ga}. This proves~(\ref{cor_item1}) and~(\ref{cor_item2}) follows immediately.
\end{proof}

\begin{remark}
\label{rem:klassen_paper}
We can define an action of $\ol \Ga$ on $L^2(I,\R^d)$ via
\[
q * \ga = (q \o \ga) \cdot \sqrt{\ga'}\,.
\]
This is a linear, isometric action and it makes the square root velocity transform equivariant,
\[
R(c \o \ga) = R(c) * \ga\,.
\]
We can then formulate Prop.~\ref{prop:orbit_closed} and Cor.~\ref{cor:orbit_id} directly on $L^2(I,\R^d)$, the space of square root velocity functions: if $q \in L^2$ and $q \neq 0$ a.e., then $q \ast \ol\Ga$ is closed and $\ol{q \ast \Ga} = q \ast \ol\Ga$. In this formulation the statement has been proven in \cite[Thm.~3]{Lahiri2015_preprint} without using the continuity of $R$ or the continuity of the $\ol\Ga$-action.
\end{remark}

\section{Shape space of unparametrised curves}
\label{sec:shape_space}

\subsection{Equivalence up to parametrisation} 
We are interested in identifying cur\-ves up to reparametrisations. Since we are working with a \emph{semi}group of re\-pa\-ra\-met\-ri\-sa\-tions, we have to be careful, when talking about orbits of the $\Ga$-action; we will use the fact that $\ol\Ga$ contains $\Ga$ as a dense subgroup. Before we define what it means for two curves to be equivalent up to reparametrisations, first a helpful lemma.

\begin{lemma}
\label{lem:orbit_choice}
Let $b,c \in AC_0(I,\R^d)$. Then 
\[
\ol{b \o \Ga} \cap \ol{c \o \Ga} = \emptyset \text{ or } \ol{b \o \Ga} = \ol{c \o\Ga}\,.
\]
Two of these sets coincide, $\ol{b \o \Ga} = \ol{c \o\Ga}$, if and only if $b$ and $c$ have the same constant speed parametrisation.
\end{lemma}

Note that if $b, c$ are regular curves, then we can rephrase the lemma in terms of $\ol\Ga$-orbits,
\[
b \o \ol\Ga \cap c \o \ol\Ga = \emptyset \text{ or } b \o \ol\Ga = c \o \ol\Ga\,,
\]
and $b\o\ol\Ga = c\o\ol\Ga$ if and only if $b$ and $c$ have the same constant speed parametrisation.

\begin{proof}
Using Cor.~\ref{cor:orbit_id} it is enough to prove the lemma for regular curves $b,c$, in which case $\ol{b \o \Ga} = b \o\ol\Ga$ and the same for $\ga$. Assume that $b \o \ol\Ga$ and $ c \o \ol\Ga$ have a nonempty intersection, i.e., $b \o \be = c \o \ga$ for some $\be, \ga \in \ol\Ga$. By choosing constant speed reparametrisations we can assume that, $|b'|\equiv \la$ and $|c'|\equiv \mu$. By taking the norm of derivative we obtain $\la \be' = \mu \ga'$ and since $\int_I \la \be' = \la$, it follows that $\la = \mu$ and $\be = \ga$. Next we approximate $\be$ by a sequence $\be_n \in \Ga$, i.e., $\be_n \to \be$ in $\ol \Ga$. We have the identity
\[
b \o \be \o \be_n\inv = c \o \be \o \be_n\inv\,,
\]
and by taking the limit we obtain $b = c$. Thus two orbits either coincide or they are disjoint.
\end{proof}

In the next proposition we define and characterise equivalence classes of unparametrised curves.

\begin{proposition}
\label{prop:eq_rel}
The following are equivalent ways to define an equivalence relation on $AC_0(I,\R^d)$.
\begin{enumerate}
\item
\label{equiv1}
$b \sim c \Leftrightarrow \exists a \in AC_0,\; \exists \be, \ga \in \ol\Ga:\;
b = a \o \be \text{ and } c = a \o \ga\,.$
\item
\label{equiv2}
$b \sim c \Leftrightarrow \ol{b \o \Ga} = \ol{c \o \Ga}$.
\item
\label{equiv3}
Denote by $\mc A \subset AC_0$ the set of constant speed curves. Then
\[
AC_0 = \{0\} \cup \bigcup_{c \in \mc A} c \o \ol\Ga
\]
is a partition of $AC_0$ into disjoint sets.
\end{enumerate}
The equivalence classes are given by $[c] = \ol{c \o \Ga}$.
\end{proposition}

Property~(\ref{equiv1}) states that two curves are equivalent if they are reparametrisations of a common curve. This curve can be taken to have constant speed, leading to the alternative characterisation
\begin{enumerate}
\item[(1')]
$b \sim c \Leftrightarrow \exists a \in \mc A,\; \exists \be, \ga \in \ol\Ga:\;
b = a \o \be \text{ and } c = a \o \ga\,.$
\end{enumerate}
Because of Lem.~\ref{lem:orbit_choice}, property~(\ref{equiv2}) is also equivalent to
\begin{enumerate}
\item[(2')]
$b \sim c \Leftrightarrow \ol{b \o \Ga} \cap \ol{c \o \Ga} \neq \emptyset$.
\item[(2'')]
$b \sim c \Leftrightarrow b \in \ol{c \o \Ga}$.
\end{enumerate}
The equivalence follows from the implications (2'')$\Rightarrow$(2')$\Rightarrow$(2)$\Rightarrow$(2''). Property~(\ref{equiv2}) is used in \cite{Lahiri2015_preprint} as the definition.

\begin{proof}[Proof of Prop.~\ref{prop:eq_rel}] We will denote by $\sim_1$, $\sim_2$, $\sim_3$ the equivalence relations of (1), (2) and (3) respectively. It is clear that $\sim_1$ is symmetric and reflexive. Transitivity will follow from identifying the equivalence classes.

(1) Fix $c \in AC_0$ and assume that $b \sim_1 c$. Then $b = a\o\be$ and $c = a \o\ga$ for some $\be, \ga \in \ol\Ga$. Choose $\ga_n \in \Ga$ with $\ga_n \to \ga$. Then $c \o \ga_n\inv \o \be = a \o (\ga \o \ga_n\inv) \o \be \to b$ by Lem.~\ref{lem:inv_conv_ga} and hence $b \in \ol{c \o \Ga}$.

Conversely, if $b \in \ol{c \o\Ga}$, then $b \in \tilde c \o \ol\Ga$, where $\tilde c$ is a constant speed parametrisation of $c$, i.e. $c = \tilde c \o \tilde \ga$. Thus $b = \tilde c \o \tilde \be$ for some $\tilde \be \in \ol \Ga$ and hence $b \sim_1 c$. Thus $[c]_1 = \ol{c \o \Ga}$.

(2) Again, fix $c \in AC_0$. If $b \sim_2 c$, then clearly $b \in \ol{c \o \Ga}$. Conversely, if $b \in \ol{c \o \Ga}$, then $\ol{b \o \Ga} \cap \ol{c \o \Ga} \neq \emptyset$ and thus $\ol{b \o \Ga} =\ol{c \o \Ga}$ by Lem.~\ref{lem:orbit_choice}. By definition this means $b \sim_2 c$ and hence $[c]_2 = \ol{c\o \Ga}$.

(3) Lemma~\ref{lem:orbit_choice} shows that the sets $c \o \ol\Ga$, where $c \in \mc A$, together with $\{0\}$ form a partition of $AC_0$. Take $b \in AC_0$ and write it as $b = c\o \ga$ with $c \in AC$ the unique constant speed parametrisation. Then $[b]_3 = c\o\ol\Ga = \ol{b\o\Ga}$ by Cor.~\ref{cor:orbit_id}.
\end{proof}

\subsection{Quotient space}
\label{sec:quotient_space}
Using the equivalence relation defined in Prop.~\ref{prop:eq_rel} we introduce the quotient space of unparametrised curves
\begin{align*}
B(I,\R^d) &:= AC_0(I,\R^d) /_\sim\,,
\end{align*}
together with the canonical projection
\[
\pi : AC_0(I,\R^d) \to B(I,\R^d),\, c \mapsto  [c]\,.
\]
In the following we will use the following representation of $B(I,\R^d)$,
\[
B(I,\R^d) = \left\{ c \o \ol\Ga \,:\, c \in AC_0(I,\R^d),\, c'\neq 0\text{ a.e.}\right\}
\cup \{0\}\,.
\]
Informally $B(I,\R^d)$ is the quotient space $AC_0/\ol{\Ga}$, however since $\ol\Ga$ is not a group, we only consider orbits of curves with non-zero derivative a.e., together with the constant curve; we identify the constant curve with the orbit $0 \o \ol\Ga$. Unless stated otherwise, statements about elements $c \o \ol\Ga \in B(I,\R^d)$ assume implicitly either $c'\neq 0$ a.e. or $c\equiv 0$.

\subsection{Induced distance}
The distance
\[
\on{dist}(b, c) = \| R(b) - R(c) \|_{L^2}
\]
on $AC_0$, defined in Sect.~\ref{sec:geod_dist}, is invariant under the $\ol\Ga$-action, as can be seen from
\begin{align*}
\on{dist}(c_1 \o \ga, c_2 \o \ga)
&= \| R(c_1 \o \ga) - R(c_2 \o \ga) \|_{L^2} \\
&= \left\| \left( R(c_1) - R(c_2)\right) * \ga \right\|_{L^2} \\
&= \| R(c_1) - R(c_2) \|_{L^2} = \on{dist}(c_1,c_2)\,;
\end{align*}
here $\ga \in \ol\Ga$ and for notational convenience we used the isometric $\ol\Ga$-action on $L^2$, introduced in Rem.~\ref{rem:klassen_paper}.

On $B(I,\R^d)$ we consider the induced quotient distance
\[
\on{dist}(b \o \ol\Ga, c \o \ol\Ga) = \inf_{\be,\ga \in \ol\Ga} \on{dist}(b \o \be, c \o \ga)
= \inf_{\ga \in \Ga} \on{dist}(b, c \o \ga)\,.
\]
Note that for the second equality to hold, we need that $\Ga$ is dense in $\ol\Ga$ and that $\ol\Ga$ acts continuously with respect to $\on{dist}$. This allows us to choose a minimising sequence $(\be_n, \ga_n)$ in $\Ga$ instead of $\ol\Ga$ and use the invariance of $\on{dist}$ to write
\[
\on{dist}(b \o \be_n, c \o \ga_n)
= \on{dist}(b, c \o \ga_n \o \be_n\inv)\,.
\]
We have the following result.

\begin{lemma}
The topology induced by $\on{dist}$ on $B(I,\R^d)$ coincides with the quotient topology and $(B(I,\R^d), \on{dist})$ is a complete metric space.
\end{lemma}

This may be unsurprising, but the proof is nontrivial. It follows closely the proof given in \cite[Lem.~6.5]{Bruveris2015}, but under slightly weaker assumptions. The analog distance under the SRVT on $L^2(I,\R^d)/\ol{\Ga}$ has been studied in \cite{Lahiri2015_preprint}.

\begin{proof}
It is clear that $\on{dist}$ is symmetric and satisfies the triangle inequality. If $\on{dist}(b \o\ol\Ga, c \o\ol\Ga) = 0$, then there exists a sequence $\ga_n \in \Ga$ with $\on{dist}(b, c \o \ga_n) \to 0$, which means $b \in \ol{c \o\Ga}$ or equivalently $b\o\ol\Ga \cap c \o\ol\Ga \neq \emptyset$. This by Lem.~\ref{lem:orbit_choice} implies $b\o\ol\Ga = c\o\ol\Ga$. Thus $\on{dist}$ is indeed a distance.

Let $O \subseteq B(I,\R^d)$ be open with respect to $\on{dist}$ and take $c \in \pi\inv(O)$. Write $c = \tilde c \o \ga$ with $\tilde c$ of constant speed. We will denote by $B(c,\ep)$ and $B(c \o \ol\Ga, \ep)$ the $\ep$-balls in $AC_0(I,\R^d)$ and $B(I,\R^d)$ respectively. Then $\pi(c) = \pi(\tilde c)$ and there exists an $\ep > 0$, such that $B(\tilde c\o\ol\Ga, \ep) \subseteq O$. We claim that $B(c,\ep) \subseteq \pi\inv(O)$. Let $b \in AC_0$ be such that $\on{dist}(b,c) < \ep$ and write $b = \tilde b \o \be$ with $\tilde b$ of constant speed. Then 
\[
\on{dist}(\tilde b \o \ol\Ga, \tilde c\o\ol\Ga) 
\leq \on{dist}(\tilde b \o\be, \tilde c \o \ga)
= \on{dist}(b,c) < \ep\,,
\]
meaning $\tilde b \o\ol\Ga \in O$ and $b \in \pi\inv(O)$. Thus $\pi\inv(O)$ is open in $AC_0$ and $O$ is open in the quotient topology.

Now let $O \subseteq B(I,\R^d)$ be open in the quotient topology, $c \o \ol\Ga \in U$ and $\ep$ such that $B(c,\ep) \subseteq \pi\inv(U)$. If $\on{dist}(b \o\ol\Ga, c \o\ol\Ga) < \ep$ for some $b$, then $\on{dist}(b\o\be, c) < \ep$ for some $\be \in \Ga$ and hence $b \o \be \in B(c,\ep)$, implying $b \o \Ga \in U$. Thus $B(c\o\Ga, \ep) \subseteq U$ and the topology induced by $\on{dist}$ coincides with the quotient topology.

Now we want to show completeness of $B(I,\R^d)$. Let $(c_n \o \ol\Ga)_{n \in \mb N}$ a Cauchy sequence. We can choose a subsequence, such that $\on{dist}(c_n \o \ol\Ga,c_{n+1}\o \ol\Ga) < 2^{-n}$ holds for all $n\in \mb N$. Next we choose representatives of the orbits with $\on{dist}(c_n,c_{n+1}) < \on{dist}(c_n \o \ol\Ga,c_{n+1}\o \ol\Ga) + 2^{-n}$. Then
\begin{align*}
\on{dist}(c_n,c_{n+k}) &\leq \sum_{i=n}^{n+k-1} \on{dist}(c_i,c_{i+1}) \\
&\leq \sum_{i=n}^{n+k-1} \on{dist}(c_i \o \ol\Ga,c_{i+1}\o \ol\Ga) + 2^{-i}
\leq 2^{2-n}(1-2^{-k})\,,
\end{align*}
showing that $(c_n)_{n \in \mb N}$ is a Cauchy sequence in $AC_0(I,\R^d)$. Let $c$ be the limit. Then $\lim c_n \o \ol\Ga = \lim \pi(c_n) = c\o\ol\Ga$ and thus $B(I,\R^d)$ is complete.
\end{proof}

\section{Existence of optimal reparametrisations}

In this section we want to answer the question, whether the infimum in 
\[
\on{dist}(b \o \ol\Ga, c \o \ol\Ga) 
= \inf_{\be,\ga \in \ol\Ga} \on{dist}(b \o \be, c \o \ga)
\]
is attained. Before stating the main result, we want to cite a theorem about upper semi continuity of functionals, that will be used in the proof.

\begin{theorem}[Thm. 1.6 in \cite{Struwe2008}]
\label{thm:upper_semicont}
Let $I$ be a compact interval and assume that $F : I \x \R^d \x \R^d \to \R$ is a continuous function and $F(t,x,\cdot)$ is concave for all $t, x$. Then, if $u_n, u \in W^{1,\infty}(I,\R^d)$ and $u_n \to u$ in $L^1(I)$, $u_n \rightharpoonup u$ weakly in $L^1(I)$ and $\| u_n' \|_{L^\infty} \leq C$ for some $C \in \R$, it follows that
\[
E(u) \geq \limsup_{n \to \infty} E(u_n)\,,
\]
where
\[
E(u) = \int_I F(t, u(t), u'(t)) \ud t\,.
\]
\end{theorem}

This is a version of \cite[Thm. 1.6]{Struwe2008}, rewritten for concave, instead of convex functions, with the boundedness assumption moved from $F$ to the sequence $u_n$.

Here is the main result.

\begin{proposition}
\label{prop:ex_opt_reparam}
Given $b,c \in C^1(I,\R^d)$ with $b', c' \neq 0$ a.e., there exist $\be,\ga \in \ol\Ga$, such that
\[
\on{dist}(b\o\be,c\o\ga) = \on{dist}(b \o \ol\Ga, c \o \ol\Ga)\,,
\]
i.e., the infimum in the definition of $\on{dist}(b \o \ol\Ga, c \o \ol\Ga)$ is attained.
\end{proposition}

\begin{proof}
Set $p = R(b)$ and $q = R(c)$. We note that since $b,c$ are $C^1$, their transforms $p,q$ are continuous; this will be important later on. We can write the distance $\on{dist}(b \o \be, c \o \ga)^2$ in the following form,
\begin{equation}
\label{eq:inf_to_sup}
\begin{aligned}
\on{dist}(b \o \be, c \o \ga)^2
&= \int_I \left| p \o \be \sqrt{\be'} - q \o \ga \sqrt{\ga'} \right|^2 \ud t \\
&= \int_I \left| p \o \be \right|^2 \be' 
- 2 \langle p \o \be, q \o \ga \rangle \sqrt{\be'} \sqrt{\ga'}
+ \left| q \o \ga \right|^2 \ga' \ud t \\
&= \| p \|^2_{L^2} + \| q \|^2_{L^2} - 2 \int_I \langle p \o \be, q \o \ga \rangle \sqrt{\be'}\sqrt{\ga'} \ud t\,.
\end{aligned}
\end{equation}
Thus, finding the infimum of $\on{dist}(b \o \be, c \o \ga)^2$ is equivalent to finding the supremum of $\int_I \langle p \o \be, q \o \ga \rangle \sqrt{\be'}\sqrt{\ga'} \ud t$.

{\bfseries Step 1:} Constructing a weakly convergent subsequence. \newline
Take a maximising sequence $\be_n, \ga_n \in \ol\Ga$. Before we can extract a weakly convergent subsequence we have to modify it slightly.

Considering the pair $(\be_n, \ga_n)$ as an element of $AC(I,\R^2)$, we write $(\be_n, \ga_n) = (f,g) \o \ph$ with $(f,g)$ a constant speed curve and $\ph \in \ol \Ga$. We have the freedom to choose the norm in which we measure speed and we choose $(f,g)$ to have constant speed with respect to the $1$-norm on $\R^2$, i.e., $|f'(t)| + |g'(t)| = f'(t) + g'(t) \equiv C$ is a.e. constant; because of $\int_I f' \ud t = \int_I g' \ud t = 1$, the constant has to equal $C=2$. Since 
\[
\on{dist}(c\o\be_n, c\o\ga_n) 
= \on{dist}(c\o f\o\ph, c\o g \o \ph) 
= \on{dist}(c\o f, c\o g)\,,
\]
we can replace $\be_n, \ga_n$ by $f,g$ and thus assume that the minimising sequence satisfies $0 \leq \be'_n,\ga'_n \leq 2$ a.e..

The sequences $(\be_n)_{n\in\mathbb N}$, $(\ga_n)_{n\in\mathbb N}$ have uniformly bounded derivatives and are therefore uniformly Lipschitz. Using the theorem of Arzel\`{a}--Ascoli we can pass to uniformly convergent subsequences $\be_n \to \be$ and $\ga_n \to \ga$ and the limits $\be, \ga$ are again Lipschitz with $0 \leq \be',\ga' \leq 2$; in particular $\be,\ga \in \ol\Ga$.

The sequences $(\be'_n)_{n\in\mathbb N}$, $(\ga'_n)_{n\in\mathbb N}$ are bounded in $L^\infty(I,\R)$, and $L^\infty$ being the dual of $L^1$, we can use the theorem of Banach--Alaoglu to pass to weak-$\ast$ convergent subsequences $\be'_n \xrightarrow{w^\ast} \si$ and $\ga'_n \xrightarrow{w^\ast} \rh$. Let $f$ be a smooth function with $f(0)=f(1)=0$. Then
\[
\int_I \si f
= \lim_{n \to \infty} \int_I \be_n' f 
= -\lim_{n \to \infty} \int_I \be_n f'
= - \int_I \be f'
= \int_I \be' f\,.
\]
This holds for all smooth functions, that vanish at the endpoints and hence we have $\si = \be'$ and by the same argument also $\rh = \ga'$. Thus $\be'_n \xrightarrow{w^\ast} \be'$ and $\ga'_n \xrightarrow{w^\ast} \ga'$ in $L^\infty(I,\R)$.

{\bfseries Step 2:} Constructing the maximum. \newline
Define the function $F: \R^2 \x \R^2 \to \R$ by
\[
F(x_1,x_2, \xi_1,\xi_2) = \sqrt{\xi_1\xi_2} \max\left(\langle p(x_1), q(x_2) \rangle, 0 \right)\,.
\]
For $x_1,x_2$ fixed, $F(x_1,x_2, \cdot,\cdot)$ is concave. By the construction in Step 1 the sequence $(\be_n)_{n \in \mathbb N}$ converges uniformly, $\be_n \to \be$, and thus also in $L^1$. The derivatives are bounded in $L^\infty$ and converge, $\be_n' \xrightarrow{w^\ast} \be'$, weak-$\ast$ in $L^\infty$ and thus also weakly in $L^1$. The same holds for $(\ga_n)_{n \in \mathbb N}$. This allows us to apply Thm.~\ref{thm:upper_semicont} to conclude that
\begin{align*}
\int_I \max\left(\langle p \o \be, q \o \ga \rangle, 0 \right)
\sqrt{\be'\ga'} \ud t & \geq
\limsup_{n \to \infty}
\int_I \max\left(\langle p \o \be_n, q \o \ga_n \rangle, 0 \right)
\sqrt{\be_n' \ga_n'} \ud t \\
& \geq
\limsup_{n \to \infty}
\int_I \langle p \o \be_n, q \o \ga_n \rangle 
\sqrt{\be_n' \ga_n'} \ud t\,.
\end{align*}

Finally we apply Lem.~\ref{lem:remodel_pair} with the pair $\be, \ga$ to obtain a second pair $\tilde \be, \tilde \ga$. Introduce the set $A = \left\{ t \,:\, \langle p \o \be(t), q \o \ga(t)\rangle \geq 0 \right\}$. The new pair satisfies
\begin{align*}
\int_I \langle p \o \tilde \be, q \o \tilde \ga \rangle 
\sqrt{\tilde \be' \tilde\ga'} \ud t
&= \int_{A} \langle p \o \be, q \o \ga \rangle 
\sqrt{\be'\ga'} \ud t \\
&= \int_I \max\left( \langle p \o \be, q \o \ga \rangle, 0 \right)
\sqrt{\be'\ga'} \ud t \\
&\geq \limsup_{n \to \infty}
\int_I \langle p \o \be_n, q \o \ga_n \rangle 
\sqrt{\be_n'\ga_n'} \ud t\,.
\end{align*}
Thus we see that $(\tilde \be, \tilde \ga)$ realises the supremum of $\int_I \langle p\o \be, q \o \ga \rangle \sqrt{\be'}\sqrt{\ga'} \ud t$ and hence the distance $\on{dist}(b \o \ol\Ga, c \o \ol \Ga)$.
\end{proof}

Informally the lemma states that we can change a given pair of reparametrisations and by doing so eliminate the negative contributions in the integral $\int_I \langle p\o \be, q \o \ga \rangle \sqrt{\be'}\sqrt{\ga'} \ud t$. In the proof of Prop.~\ref{prop:ex_opt_reparam} we used the calculus of variations to maximise the positive contributions and this lemma tells us, that we can remove the negative ones by hand.

\begin{lemma}
\label{lem:remodel_pair}
Let $p,q\in C(I,\R^d)$ and $\be, \ga \in \ol \Ga$. Then there exist $\tilde \be, \tilde \ga \in \ol \Ga$, such that
\[
\int_I \langle p \o \tilde \be, q \o \tilde \ga \rangle 
\sqrt{\tilde \be'}\sqrt{\tilde \ga'} \ud t
= 
\int_A \langle p \o \be, q \o \ga \rangle \sqrt{\be'}\sqrt{\ga'} \ud t\,,
\]
where $A = \left\{ t\,:\, \langle p \o \be(t), q \o \ga(t) \rangle \geq 0\right\}$.
\end{lemma}

\begin{proof}
Since $p,q$ are continuous, the set $B = \left\{ t\,:\, \langle p \o \be(t), q \o \ga(t) \rangle < 0\right\}$ is open and thus can be written as an at most countable union, $B = \bigcup_{n} I_n$ of disjoint open intervals, $I_n = [t_n^-, t_n^+]$. We define the new parametrisations $\tilde \be$, $\tilde \ga$ as follows: we set $\tilde \be|_A = \be|_A$, $\tilde \ga|_A = \ga|_A$; to define them on $B$ we split each interval into $I_n = I_n^- \cup I_n^+$ with $I_n^- = [t_n^-, \frac 12 (t_n^- + t_n^+)]$ and $I_n^+ = [\frac 12 (t_n^- + t_n^+), t_n^+]$ and set
\begin{align*}
\tilde \be' &= \begin{cases}
2\be'(2t - t_n^-) & t \in I_n^- \\
0 & t \in I_n^+
\end{cases} &
\tilde \ga' &= \begin{cases}
0 & t \in I_n^- \\
2\ga'(2t - t_n^+) & t \in I_n^+
\end{cases}\,.
\end{align*}
When integrating $\tilde \be'$ we choose $\tilde \be(t_n^-) = \be(t_n^-)$ as the constant of integration and this choice leads to $\tilde \be(t_n^+) = \be(t_n^+)$. Thus $\tilde \be$ is again absolutely continuous. Furthermore we have the property $\sqrt{\tilde \be'}\sqrt{\tilde \ga'} = 0$ on $I_n$ and hence also on $B$. Together we obtain
\[
\int_I \langle p \o \tilde \be, q \o \tilde \ga \rangle 
\sqrt{\tilde \be'}\sqrt{\tilde \ga'} \ud t
=
\int_A \langle p \o \tilde \be, q \o \tilde \ga \rangle 
\sqrt{\tilde \be'}\sqrt{\tilde \ga'} \ud t
= 
\int_A \langle p \o \be, q \o \ga \rangle \sqrt{\be'}\sqrt{\ga'} \ud t\,,
\]
as required.
\end{proof}

\subsection{Counterexample}
In this example we will construct a pair of Lipschitz curves in the plane, for which no optimal re\-pa\-ra\-met\-ri\-sa\-tions exist. This shows that some additional assumption on the regularity of the curves -- for example $C^1$ as in Prop.~\ref{prop:ex_opt_reparam} -- is necessary for the existence of optimal reparametrizations.

Let $B\subset I$ be a modified Cantor set with the following properties: $B$ is closed and nowhere dense and $\la(B) = \frac 12$ with $\la$ denoting the Lebesgue measure; see \cite[Ex. I.1.7.6]{Bogachev2007} for the construction of $B$. Setting $A = I \setminus B$, we have that $A$ is open and dense and $\la(A) = \la(B) = \frac 12$.
 
We choose a curve $v_1(t) \in \R^2$ and vectors $v_2, v_3 \in \R^2$ as follows
\begin{align*}
v_1 &= \begin{pmatrix} 
\cos \ep t  \\
\sin \ep t 
\end{pmatrix} &
v_2 &= \begin{pmatrix}
-\frac 12 \\
\frac{\sqrt{3}}2
\end{pmatrix} &
v_3 &= \begin{pmatrix}
-\frac 12 \\
-\frac{\sqrt{3}}2
\end{pmatrix}\,,
\end{align*}
with $\ep < \frac 1 6$ a small number. The three vectors $v_1(t), v_2, v_3$ have the property that all mixed scalar products are negative, $\langle v_1(t), v_2\rangle < 0$, $\langle v_2, v_3\rangle < 0$ and $\langle v_1(t), v_3\rangle < 0$. We define the two curves
\begin{align*}
p(t) &= v_1(t) \one_A(t) + v_2 \one_B(t) \\
q(t) &= v_1(t) \one_A(t) + v_3 \one_B(t)\,.
\end{align*}
We have $p, q \in L^\infty(I,\R^2)$ and thus their preimages $b=R\inv(p)$, $c = R\inv(q)$ are well-defined Lipschitz curves, hence also absolutely continuous. We claim that the infimum $\inf_{\be, \ga \in \ol\Ga} \on{dist}(b \o \be, c \o \ga)$ is not attained. Because of
\[
\on{dist}(b \o \be, c \o \ga)^2
= \| p \|^2_{L^2} + \| q \|^2_{L^2} - 2 \int_I \langle p \o \be, q \o \ga \rangle \sqrt{\be'}\sqrt{\ga'} \ud t
\]
it is enough to look at the supremum over $\int_I \langle p \o \be, q \o \ga \rangle \sqrt{\be'}\sqrt{\ga'} \ud t$ and the next proposition shows that this supremum is not attained.

\begin{proposition}
\label{prop:ex_sup}
With $p,q$ constructed as above we have
\begin{equation}
\label{eq:int_sup}
\sup_{\be, \ga \in \ol \Ga} \int_I \langle p \o \be, q \o \ga \rangle \sqrt{\be'} \sqrt{\ga'} \ud t = \la(A)\,,
\end{equation}
and the supremum is not attained.
\end{proposition}

\begin{proof}
{\bfseries Step 1:} $\sup \leq \la(A)$. \newline
Because all mixed scalar products among $v_1(t), v_2, v_3$ are negative, we have the simple estimates
\begin{align*}
\int_I \langle p \o \be, q \o \ga \rangle \sqrt{\be'} \sqrt{\ga'} \ud t &\leq
\int_{\be\inv(A) \cap \ga\inv(A)} \langle v_1 \o \be, v_1 \o \ga \rangle \sqrt{\be'} \sqrt{\ga'} \ud t \\
&= \int_{\be\inv(A) \cap \ga\inv(A)} \cos \ep  (\ga(t) - \be(t)) \sqrt{\be'} \sqrt{\ga'} \ud t \\
&\leq \int_{\be\inv(A) \cap \ga\inv(A)} \sqrt{\be'} \sqrt{\ga'} \ud t \,.
\end{align*}
We set $M = \be\inv(A) \cap \ga\inv(A)$ and since $M$ is an open set, we can write $M = \bigcup_j I_j$ as a union of countably many disjoint open intervals $I_j$. Using the inequality of Cauchy--Schwartz a couple of times we obtain
\begin{align*}
\int_M \sqrt{\be'} \sqrt{\ga'} \ud t 
&= \sum_j \int_{I_j} \sqrt{\be'} \sqrt{\ga'} \ud t \\
&\leq \sum_j \sqrt{\int_{I_j} \be' \ud t}
\sqrt{\int_{I_j} \ga' \ud t}
= \sum_j \sqrt{\la(\be(I_j))} \sqrt{\la(\ga(I_j))} \\
&\leq \sqrt{\sum_j \la(\be(I_j))} \sqrt{\sum_j \la(\ga(I_j))}
= \sqrt{\la(\be(M))} \sqrt{\la(\ga(M))}
\leq \la(A)\,.
\end{align*}
In particular we see that the upper bound
\[
\int_I \langle p \o \be, q \o \ga \rangle \sqrt{\be'} \sqrt{\ga'} \ud t \leq \la(A)
\]
holds for all reparametrisations $\be, \ga$.

{\bfseries Step 2:} $\la(A)$ is not attained. \newline
Let $\be, \ga \in \ol \Ga$ be a pair of reparametrisations. As in the proof of Prop.~\ref{prop:ex_opt_reparam} we can replace the curve $(\be, \ga) \in AC(I,\R^2)$ by its $L^1$-constant speed parametrisation allowing us to assume that $\be'+\ga'=2$ holds a.e..

By following the estimates made in Step 1, we see that a necessary condition for the equality
\begin{equation}
\label{eq:simple_sup}
\int_M \langle v_1 \o \be, v_1 \o \ga \rangle \sqrt{\be'} \sqrt{\ga'} \ud t = \la(A)
\end{equation}
to hold is $\be' = r \ga'$ on $M$ for some $r \in \R$. Together with $\be'+\ga'=2$ this implies that $\sqrt{\be'}\sqrt{\ga'} \neq 0$ a.e. on $M$ and thus $\cos \ep (\ga(t) - \be(t)) = 1$ a.e. on $M$. Hence $\be|_M = \ga|_M$. Since $\be$ is a closed map we have $\be(\ol M) \supseteq \ol{\be(M)} = \ol A = I$. Hence $\be(\ol M) = I$ and because $\be|_M = \ga|_M$, also $\ga(\ol M) = I$. 

Assume $M$ is not dense in $I$. Then $\ol M^c$ contains an open interval $O$ and because $\be$ is weakly increasing and $\be(\ol M) = I$, $\be$ must be constant on $O$, in particular $\be'|_O = 0$. The same holds for $\ga$, $\ga'|_O =0$, but this contradicts the assumption $\be' + \ga' = 2$ a.e.; hence $\ol M = I$.

By continuity, $\be|_M = \ga|_M$ implies $\be = \be|_{\ol M} = \ga|_{\ol M} = \ga$ and hence $\be' = \ga'$ a.e.. Because of $\be' + \ga' = 2$ we actually have $\be' = \ga' = 1$ and thus $\be(t) = t$ and $\ga(t) = t$. The scalar products $\langle v_i, v_j \rangle$ were chosen to be negative for $i \neq j$, therefore the equality 
\[
\int_I \langle p \o \be, q \o \ga \rangle \sqrt{\be'} \sqrt{\ga'} \ud t
= \int_M \langle p \o \be, q \o \ga \rangle \sqrt{\be'} \sqrt{\ga'} \ud t
\]
can only hold if $\sqrt{\be'}\sqrt{\ga'} = 0$ on $I \setminus M$. Together with $\be' = \ga' = 1$, this leads to a contradiction. Hence the value $\la(A)$ cannot be attained.

{\bfseries Step 3:} $\sup \geq \la(A)$. \newline
We will construct a sequence of reparametrisations, such that the integral in~\eqref{eq:int_sup} will converge to $\la(A)$.

Since $B$ is measurable, there exists a sequence of open sets, $O_n \supseteq B$, such that $\la(O_n \setminus B) \to 0$. Decompose $O_n = \bigcup_k I_{n,k}$ into at most countably many open intervals. Each interval $I_{n,k}$ shall be divided into two subintervals of equal size, $I_{n,k} = I_{n,k}^- \cup I_{n,k}^+$. We define the reparametrisations $\be_n, \ga_n$ by setting $\be_n(t) = \ga_n(t) = t$ for $t \in O_n^c$. On $O_n$ we define
\begin{align*}
\be_n'&=\begin{cases}
2 & \text{on }I_{n,k}^- \\
0 & \text{on }I_{n,k}^+
\end{cases} &
\ga_n'&=\begin{cases}
0 & \text{on }I_{n,k}^- \\
2 & \text{on }I_{n,k}^+
\end{cases}\,.
\end{align*}
This has the effect that $\be_n, \ga_n$ are continuous and $\sqrt{\be_n'} \sqrt{\ga_n'} = 0$ on $O_n$.

Now we look at the integral. Since $B \subseteq O_n$, it follows that $O_n^c \subseteq A$. Thus
\[
\int_I \langle p \o \be, q \o \ga \rangle \sqrt{\be'} \sqrt{\ga'} \ud t = 
\int_{O_n^c} \langle p \o \be, q \o \ga \rangle \sqrt{\be'} \sqrt{\ga'} \ud t = \la(O_n^c)\,.
\]
Comparing to the desired value $\la(A)$ we see that
\[
\la(A) - \la(O_n^c) = \la(A) - \la(A \cap O_n^c)  = \la(A \cap O_n) = \la(O_n \cap B^c) = \la(O_n \setminus B) \to 0\,,
\]
meaning that we can approximate $\la(A)$ by a sequence of reparametrisations. Thus $\sup \geq \la(A)$, which concludes the proof.
\end{proof}

We summarise the counterexample in the following corollary.

\begin{corollary}
\label{cor:count_ex}
Let $d \geq 2$. Then there exist two curves, $b, c \in W^{1,\infty}(I, \R^d)$, such that the infimum
\[
\inf_{\be, \ga \in \ol \Ga} \on{dist}(b \o \be, c \o \ga)
\]
is not attained.
\end{corollary}

We would like to contrast this result to \cite[Thm.~4]{Lahiri2015_preprint}, which states that if one curve is piecewise linear then the other curve only has to be absolutely continuous for the infimum to be attained. We do not know, if we can strengthen the counterexample to make one curve $C^1$ while the other one remains Lipschitz. The above construction cannot be immediately generalised to scalar functions and thus the case $d=1$ remains open.

\begin{openquestion*}
Does there exist a pair of scalar functions $b, c \in AC_0(I,\R)$, such that the infimum 
$\inf_{\be, \ga \in \ol \Ga} \on{dist}(b \o \be, c \o \ga)$ is not attained?
\end{openquestion*}

Both the proof of existence of optimal reparametrisations as well as the construction of the counterexample relied heavily on the availability of an explicit formula for the geodesic distance. Such a formula is not available for closed curves and hence the case of periodic functions remains open.

\begin{openquestion*}
Can Prop.~\ref{prop:ex_opt_reparam} and Cor.~\ref{cor:count_ex} be generalised to the metric completion of the geodesic distance on the space $AC_0(S^1,\R^d)$ of closed curves?
\end{openquestion*}

\subsection*{Acknowledgements}
I would like to thank Martin Bauer, Philipp Harms, Eric Klassen, Stephen Marsland and Peter W. Michor for their valuable comments and helpful discussions.

\bibliographystyle{siamplain}

\end{document}